\documentclass[11pt]{article}
\usepackage[margin=1in]{geometry}
\usepackage[utf8]{inputenc}
\usepackage[T1]{fontenc}
\usepackage[english]{babel}
\usepackage{lmodern}
\usepackage{graphicx}
\usepackage{wrapfig}
\usepackage{subfigure}
\usepackage{caption}
\usepackage{paralist}
\usepackage{enumitem}

\usepackage[linesnumbered,ruled,vlined]{algorithm2e}

\graphicspath{{./figures/}}

\makeatletter
\renewcommand{\p@enumii}{}
\makeatother

\newcommand{\titel}{Defective acyclic colorings of planar graphs}
\usepackage[usenames]{color}
\definecolor{hellblau}{rgb}{0.2,0.4,1} 
\definecolor{dunkelblau}{rgb}{0,0,0.8}
\definecolor{dunkelgruen}{rgb}{0,0.5,0}
\usepackage[
pdftex,
colorlinks,
linkcolor=dunkelblau,
urlcolor=dunkelblau,
citecolor=dunkelgruen,
bookmarks=true,
linktocpage=true,
pdfauthor={On-Hei Solomon Lo},
pdfsubject={}
]{hyperref}
\usepackage{pdfpages}
\usepackage{amsmath}
\usepackage{amsthm}
\allowdisplaybreaks
\usepackage{amsfonts}
\theoremstyle{plain}
\newtheorem{satz}{Satz}[]
\newtheorem{theorem}[satz]{Theorem}
\newtheorem{lemma}[satz]{Lemma}

\newtheorem{proposition}[satz]{Proposition}
\newtheorem{corollary}[satz]{Corollary}	

\theoremstyle{remark}

\theoremstyle{definition}
\newtheorem{definition}[satz]{Definition}

\begin{document}
	\title{\titel}
	\author{
		On-Hei Solomon Lo\thanks {Faculty of Environment and Information Sciences, Yokohama National University, Yokohama 240-8501, Japan} \and
		Ben Seamone\thanks {Mathematics Department, Dawson College, Montreal, QC, Canada}\;\thanks{D\'epartement d'informatique et de recherche op\'erationnelle, Universit\'e de Montr\'eal, Montreal, QC, Canada} \and
		Xuding Zhu\thanks {Department of Mathematics, Zhejiang Normal University, Jinhua 321004, China}\\
	}
	\date{}
	\maketitle
	
	\begin{abstract}
		This paper studies two variants of defective acyclic coloring of planar graphs. For a graph $G$ and a coloring $\varphi$ of $G$, a 2CC transversal is a subset $E'$ of $E(G)$ that intersects every 2-colored cycle. Let $k$ be a positive integer. We denote by $m_k(G)$  the minimum integer $m$ such that $G$ has a proper $k$-coloring which has a 2CC transerval of size $m$, and by $m'_k(G)$ the minimum size of a subset $E'$ of $E(G)$ such that $G-E'$ is acyclic  $k$-colorable. We prove that for any $n$-vertex $3$-colorable planar graph  $G$, $m_3(G) \le n - 3$ and for any planar graph $G$, $m_4(G) \le n - 5$ provided that $n \ge 5$. We show that these upper bounds are sharp: there are infinitely many planar graphs attaining these upper bounds. Moreover, the minimum 2CC transversal $E'$  can be chosen in such a way that $E'$ induces a forest. We also prove that for any planar graph $G$, $m'_3(G) \le (13n - 42) / 10$ and $m'_4(G) \le (3n - 12) / 5$.
	\end{abstract}
	
	\section{Introduction} \label{sec:intro}
	
	An {\em acyclic $k$-coloring} of a graph $G$ is  a proper $k$-coloring of  $G$ with no 2-colored cycles.  Confirming a conjecture of Gr{\"u}nbaum~\cite{Gruenbaum1973},   Borodin~\cite{Borodin1979} proved   that every planar graph has an acyclic 5-coloring. This celebrated result is best possible as there are planar graphs that are not acyclic 4-colorable (e.g.\ the octahedron). Acyclic coloring has been  studied extensively for several decades and applied to solve other problems on graph coloring and partitioning. We refer to~\cite{Borodin2013} for a comprehensive survey on this subject.
	
	This paper studies defective acyclic $k$-coloring of planar graphs mainly for $k=3,4$.
	In other words, we study $k$-colorings of planar graphs for which the condition of being an acyclic coloring is not completely satisfied, however, we want to limit the violation of the acyclicity rules. We consider two variants of defective acyclic coloring. 
	
	\begin{definition}
		\label{def-transversal}
		Given a graph $G$ and a  proper coloring $\varphi$ of $G$, a {\em $2$-colored cycle transversal} ($2$CC transversal) with respect to $\varphi$ is a subset $E'$ of $E(G)$ that intersects all 2-colored cycles. In other words, $G-E'$ contains no 2-colored cycles.  
	\end{definition}

	\begin{definition}
		Let $G$ be a graph and $k$ be a positive integer. We define two parameters $m_k(G)$ and $m'_k(G)$ as follows:
		\begin{itemize}
			\item $m_k(G):=\min_{E' \subseteq E(G)}\{|E'|: \text{$E'$ is a 2CC transversal with respect to a proper $k$-coloring}\}.$
			\item    $m'_k(G):=\min_{E' \subseteq E(G)}\{|E'|: \text{$G-E'$ has an acyclic $k$-coloring}\}.$
		\end{itemize}
	\end{definition}
	
	Note that $m_k(G)=m'_k(G) =0$ if and only if $G$ is acyclic $k$-colorable. If $G$ has no proper $k$-coloring, then   $m_k(G)$ is not defined.  In this case, we let $m_k(G) := \infty$. It follows from the definition that for any graph $G$ and integer $k$,   $m_k(G) \ge m'_k(G)$.  
	
	We are interested in the case that $G$ is a planar graph and $k=3,4$ as Borodin's theorem asserts that $m_5(G) = 0$. To obtain an upper bound for $m_k(G)$, we need to construct a proper $k$-coloring $\varphi$ of $G$ and find a 2CC transerval $E'$. One immediate difficulty  is that, for $k=4$, the existence of a proper $4$-coloring of a planar graph follows from the Four Color Theorem. For $k=3$, it is NP-complete to decide whether a planar graph $G$ is 3-colorable, and hence there is no easy way to construct a proper 3-coloring of $G$. Fortunately, it turns out that  tight upper bounds for $m_4(G)$ and $m_3(G)$ for the whole family of  planar graphs and the whole family of 3-colorable planar graphs do not depend on a particular proper coloring of $G$. 
	
	For any proper coloring $\varphi$ of a graph $G$, define \begin{align*}
		m(G, \varphi) := \min_{E' \subseteq E(G)}\{|E'|:\text{$E'$ is a 2CC transerval with respect to $\varphi$} \}.
	\end{align*}
	We prove in Section~\ref{sec:subgraph} that for any planar graph $G$ on $n$ vertices and any proper coloring $\varphi$ of $G$, $m(G, \varphi) \le n - |\varphi(V(G))|$, where $|\varphi(V(G))|$ denotes the number of colors used in $\varphi$. To this end, we study the case when $G$ is a plane triangulation in Section~\ref{sec:m}. Moreover, we show that if $n \ge 5$, then there is a 4-coloring $\varphi$ of $G$ with $m(G, \varphi) \le n - 5$. We apply these results to prove that for every planar graph $G$, $m_4(G) \le n-5$ provided that $n\ge 5$, and $m_3(G) \le n-3$ provided that $G$ is 3-colorable. These two bounds are tight as there are infinitely many 3-colorable planar graphs $G$ with $m_3(G)=n-3$ and infinitely many planar graphs $G$ with $m_4(G)=n-5$.
	Besides, we show in Section~\ref{sec:subgraph} that for any proper coloring $\varphi$ of a planar graph $G$, we can find a 2CC transerval $E'$ with $|E'| = m(G, \varphi)$ that induces a forest.
	In Section~\ref{sec:m_k} we study the parameter $m'_k(G)$. We show that $m'_3(G) \le (13n - 42) / 10$ and $m'_4(G) \le (3n - 12) / 5$.
	
	We shall mention an application of our results on acyclic colorings of subdivisions. For a graph $G$ and a positive integer $k$, define $m''_k(G)$ to be the minimum size of an edge set $E' \subseteq E(G)$ such that the graph obtained from $G$ by subdividing each edge in $E'$ by one vertex is acyclically $k$-colorable. It is easy to observe that $m_k(G) \ge m''_k(G) \ge m'_k(G)$. It was shown in \cite{MNRW2013} that for any $n$-vertex planar graph $G$, $m''_4(G) \le n - 3$. Our upper bound for $m_4(G)$ immediately improves it to $m''_4(G) \le n - 5$ for $n \ge 5$.

	All graphs considered in this paper are finite and simple. We denote by $V(G)$ and $E(G)$ the vertex set and the edge set of   $G$, respectively. For $v \in V(G)$, denote by $N_G(v)$ the set of vertices adjacent to $v$ and by $d_G(v)$ the degree of $v$. For a positive integer $k$, denote $[k] := \{1, \dots, k\}$. A \emph{$k$-coloring} $\varphi$ of $G$ is a function which assigns a color $\varphi(v) \in [k]$ to each vertex $v \in V(G)$. We say a coloring $\varphi$ is \emph{proper} if $\varphi(u) \neq \varphi(v)$ for any $uv \in E(G)$. In fact, we always consider proper colorings unless specified otherwise. Given a $k$-coloring $\varphi$ of $G$, we define the color classes by $\varphi^{-1}(i) := \{v \in V(G) : \varphi(v) = i\}$ for any $i \in [k]$.
	For any distinct $i, j \in [k]$, define $G_{ij}$ to be the subgraph of $G$ induced by $\varphi^{-1}(i) \cup \varphi^{-1}(j)$.

	\section{Upper bounds for $m(G,\varphi)$} \label{sec:m}
	
	In this section we prove upper bounds on the parameter $m(G, \varphi)$ for planar graphs. We first present several lemmas for plane triangulations.
	
	\begin{definition}
		Let $G$ be a plane triangulation on at least 4 vertices. Denote by $\mathcal{E}_G$ the set of separating triangles of $G$, and by $\mathcal{V}_G$ the set of maximal connected subgraphs of $G$ without separating triangles. The graph $\mathcal{T}_G$ is defined to be the graph on $\mathcal{V}_G$ with edge set $\mathcal{E}_G$ such that $G_1, G_2 \in \mathcal{V}_G$ are joined by $T \in \mathcal{E}_G$ if and only if both $G_1$ and $G_2$ contain $T$.
	\end{definition}
	
	It is easy to see that $\mathcal{V}_G$ is a family of 4-connected plane triangulations and $\mathcal{T}_G$ is a tree. Let $\mathcal{V}_G := \{G_1, \dots, G_t\}$ and $\mathcal{E}_G := \{T_1, \dots, T_{t-1}\}$. The graph $G$ can be retrieved from the vertex-disjoint union of $G_1, \dots, G_t$ by identifying the copies of triangle $T$ in $G_i, G_j$ for each $T = G_i G_j \in \mathcal{E}_G$. Hence $\sum_{i \in [t]} |V(G_i)| = |V(G)| + 3(t - 1)$.
	
	\begin{lemma} \label{lem:A}
		Let $G$ be a graph and $\varphi$ be a proper coloring of $G$. If $A$ is an edge set of $G$ such that $A \cap E(G_{ij})$ is an acyclic edge set for any distinct $i, j \in [k]$, then there exists $E' \subseteq E(G) \setminus A$ satisfying that $|E'| = m(G, \varphi)$ and $\varphi$ is an acyclic coloring of $G - E'$.
	\end{lemma}
	\begin{proof}
		Let $E' \subseteq E(G)$ be such that $|E'| = m(G, \varphi)$, $\varphi$ is an acyclic coloring of $G - E'$ and, subject to this, $|E' \cap A|$ is minimum. Suppose there exists $uv \in E' \cap A$. There is precisely one cycle $C$ in $G_{\varphi(u)\varphi(v)} - (E' - uv)$. As $A \cap E(G_{\varphi(u)\varphi(v)})$ is acyclic, there exists $e' \in E(C) \setminus A$. Then $G_{\varphi(u)\varphi(v)} - (E' - uv + e')$ is acyclic, $|E' - uv + e'| = |E'| = m(G, \varphi)$ and $|(E' - uv + e') \cap A| < |E' \cap A|$, contradicting our choice of $E'$. Hence $E' \subseteq E(G) \setminus A$ as desired.
	\end{proof}
	
	\begin{lemma} \label{lem:trisep}
		Let $G$ be a plane graph, $T$ be a separating triangle of $G$ and $\varphi$ be a proper coloring of $G$. Let $A_1$ and $A_2$ be the components of $G - T$, and for $i \in [2]$, $G^i$ be the subgraph of $G$ induced by $V(A_i) \cup V(T)$. Then $m(G, \varphi) = m(G^1, \varphi^1) + m(G^2, \varphi^2)$, where $\varphi^i$ denotes the restriction of $\varphi$ on $V(G^i)$.
	\end{lemma}
	\begin{proof}
		Without loss of generality, we let $V(T) = \{v_1,v_2,v_3\}$ with $\varphi(v_i) = i$ for $i \in [3]$. By Lemma~\ref{lem:A}, there exists $E' \subseteq E(G) \setminus E(T)$ such that $|E'| = m(G, \varphi)$ and $\varphi$ is an acyclic coloring of $G - E'$. As $G^i - (E' \cap E(G^i))$ is acyclically colored by $\varphi^i$ ($i \in [2]$), we have $m(G, \varphi) = |E'| = |E' \cap E(G^1)|+|E' \cap E(G^2)| \ge m(G^1, \varphi^1) + m(G^2, \varphi^2)$.
		
		Similarly, by Lemma~\ref{lem:A}, let $E_i' \subseteq E(G^i) \setminus E(T)$ be such that $|E_i'| = m(G^i, \varphi^i)$ and $G^i - E_i'$ is acyclically colored by $\varphi_i$. Let $E' := E_1' \cup E_2'$. Observe that if there is a cycle $C$ which is colored by only two colors in $G - E'$, then $C$ must contain two vertices of $T$, say $v_1, v_2$, and $C + v_1v_2$ contains some cycle in $G^1 - E_1'$ or $G^2 - E_2'$ which uses only two colors as well, a contradiction. Hence $G - E'$ is acyclically colored and $m(G, \varphi) \le |E'| = |E_1'| + |E_2'| = m(G^1, \varphi^1) + m(G^2, \varphi^2)$.
	\end{proof}
	
	\begin{lemma} \label{lem:tritree}
		Let $G$ be a plane triangulaion on at least $4$ vertices and $\varphi$ be a proper coloring of $G$. Let $\mathcal{V}_G := \{G_1, \dots, G_t\}$. We have $m(G, \varphi) = \sum_{i \in [t]} m(G_i, \varphi_i)$, where $\varphi_i$ denotes the restriction of $\varphi$ on $V(G_i)$.
	\end{lemma}
	\begin{proof}
		We prove by induction on $|\mathcal{V}_G|$. It trivially holds when $|\mathcal{V}_G| = 1$. 
		
		Suppose $|\mathcal{V}_G| > 1$. Let $T \in \mathcal{E}_G$, $A_1$ and $A_2$ be the components of $G - T$, and for $i \in [2]$, $G^i$ be the subgraph of $G$ induced by $V(A_i) \cup V(T)$. We may assume $G_1,\dots,G_{t'} \subseteq G^1$ and $G_{t'+1},\dots,G_t \subseteq G^2$ for some $1 \le t' < t$. Then, by Lemma~\ref{lem:trisep} and the induction hypothesis, $m(G, \varphi) = m(G^1, \varphi^1) + m(G^2, \varphi^2) = \sum_{i \in [t']} m(G_i, \varphi_i) + \sum_{i \in [t] \setminus [t']} m(G_i, \varphi_i) = \sum_{i \in [t]} m(G_i, \varphi_i)$.
	\end{proof}

	\begin{lemma} \label{lem:n-3}
		Let $G$ be a $3$-colorable plane triangulation on $n$ vertices and $\varphi$ be the unique proper $3$-coloring of $G$. For any distinct $i, j \in [3]$, $G_{ij}$ is connected. Moreover, if $n > 3$, $G_{ij}$ is $2$-connected.
	\end{lemma}
	\begin{proof}
		We prove by induction on $n$. The triangulations of order at most 6 are listed in Figure~\ref{fig:smalltri}. Among these graphs, only the triangle and the octahedron are 3-colorable. It is not hard to verify that the claims hold for these two graphs. From now on we assume that $n > 6$.
		
		As $G$ is a 3-colorable triangulation, every vertex of $G$ has an even degree, and hence there exists $v \in V(G)$ with $d_G(v) = 4$. Let $v_1 v_2 v_3 v_4 v_1$ be the cycle induced by $N_G(v)$. We have $\varphi(v_i) = \varphi(v_{i + 2})$ for each $i \in [2]$. 
		Suppose there exists $i \in [2]$ such that $v_i$ and $v_{i + 2}$ have no common neighbor other than $v, v_{i + 1}, v_{i + 3}$, where $v_5 := v_1$. We contract $v_i v v_{i + 2}$ to obtain $G'$ and call the new vertex $v'$. Let $\varphi': V(G') \rightarrow [3]$ be such that $\varphi'(v') = \varphi(v_i)$ and $\varphi'(u) = \varphi(u)$ for $u \in V(G') \setminus \{v'\}$. It is clear that $\varphi'$ is the unique proper 3-coloring of the triangulation $G'$. By the induction hypothesis, $G_{ij}'$ is 2-connected for any distinct $i, j \in [3]$. Then, one can easily prove by the construction that $G_{ij}$ is 2-connected for any distinct $i, j \in [3]$.
		
		Suppose for every $i \in [2]$, $v_i$ and $v_{i + 2}$ have some common neighbor other than $v, v_{i + 1}, v_{i + 3}$. Since $G$ is not the octahedron, it has some separating triangle $T$. Let $A_1, A_2$ be the components of $G - T$. We consider the subgraphs $G^i$ of $G$ induced by $V(A_i) \cup V(T)$ ($i \in [2]$). Let $\varphi_i$ be restriction of $\varphi$ on $V(G^i)$. As $|V(G^i)| > 3$, it follows from the induction hypothesis that $G_{jk}^i$ is 2-connected for any distinct $j, k \in [3]$ ($i \in [2]$), from which it immediately follows that $G_{jk}$ is 2-connected for any distinct $j, k \in [3]$.
	\end{proof}

	\begin{figure}[!ht]
		\centering
		\includegraphics[scale=1.2]{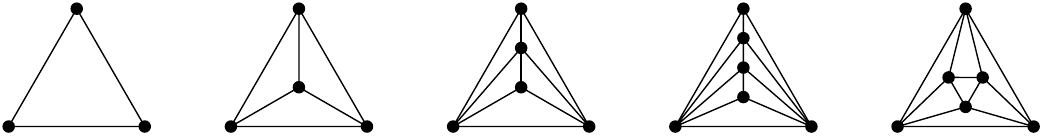}
		\caption{The triangulations of order at most 6.}
		\label{fig:smalltri}
	\end{figure}
	
	Let $G$ be a graph with a proper $k$-coloring $\varphi$. Denote by $c_{ij}$ the number of connected components of $G_{ij}$. The number of edges we need to remove from $G_{ij}$ to make $\varphi$ acyclic is $|E(G_{ij})|-|V(G_{ij})|+c_{ij}$. As $E(G_{ij})$ are edge-disjoint for distinct $i,j$, and each vertex $v$ of $G$ is contained in $k-1$ subgraphs $G_{ij}$, we know that \begin{align*}
		m(G, \phi) = \sum_{1 \le i < j \le k} (|E(G_{ij})|-|V(G_{ij})|+c_{ij}) = |E(G)|-(k-1)|V(G)|+\sum_{1 \le i < j \le k}c_{ij}.
	\end{align*} We obtain the following result by this observation.
	
	\begin{theorem} \label{thm:n-3}
		Assume $G$ is a $3$-colorable plane triangulation on $n$ vertices and $\varphi$ is the unique proper $3$-coloring of $G$. Then $m(G, \varphi) = n - 3$. For $v \in V(G)$, let $\varphi_v$ be the $4$-coloring of $G$ defined as $\varphi_v(v) = 4$ and $\varphi_v(u) = \varphi(u)$ for all $u \in V(G) \setminus \{v\}$. If $n > 3$, we have $m(G, \varphi_v) \le n - 5$.
	\end{theorem}
	\begin{proof}
		By Lemma~\ref{lem:n-3}, $G_{ij}$ is connected for any distinct $i, j \in [3]$. Hence \begin{align*}
			m(G, \varphi) &= \sum_{1 \le i < j \le 3} (|E(G_{ij})| - |V(G_{ij})| + 1) = |E(G)| - 2|V(G)| + 3 = n - 3.
		\end{align*}
		
		For the second statement, we fix $v \in V(G)$ and focus on the coloring $\varphi_v$. Without loss of generality, assume $\varphi(v) = 3$. By Lemma~\ref{lem:n-3}, $G_{12}$ (with respect to the coloring $\varphi_v$) is 2-connected. Moreover, for $i \in [2]$, the subgraph induced by $\varphi_v^{-1}(i) \cup \varphi_v^{-1}(3) \cup \{v\} = \varphi^{-1}(i) \cup \varphi^{-1}(3)$ is 2-connected and hence $G_{i3}$ (with respect to the coloring $\varphi_v$) is connected. It is also obvious that $G_{i4}$ is a forest for every $i \in [3]$. As $d_G(v) \ge 4$, we have that \begin{align*}
			m(G, \varphi) &= \sum_{1 \le i < j \le 3} (|E(G_{ij})| - |V(G_{ij})| + 1) = (|E(G)| - d_G(v)) - 2(|V(G)| - 1) + 3 \le n - 5. \qedhere
		\end{align*}
	\end{proof}
	
	We are now ready to prove the main result of this section.
	
	\begin{theorem} \label{thm:n-4}
		Assume $G$ is a  plane triangulation on $n$ vertices and $\varphi$ is a proper coloring of $G$. Let $k := |\varphi(V(G))|$. Then $m(G, \varphi) \le n - k$.
		If, in addition, $k = 4$, $n \ge 5$ and $G$ is $4$-connected, then $m(G, \varphi) \le n - 5$.
	\end{theorem}
	\begin{proof}
		We prove both statements by induction on $n$. It is easy to check that they hold for $n \le \max\{6, k\}$, thus we assume $n > \max\{6, k\}$.
		
		We first consider, for the first statement, that $G$ is not 4-connected, i.e.\ $G$ has some separating triangle $T$. Let $A_1, A_2$ be the components of $G - T$. Let $G_i$ be the subgraphs of $G$ induced by $V(A_i) \cup V(T)$ ($i \in [2]$). 
		Denote by $\varphi_i$ the restriction of $\varphi$ on $V(G_i)$. 
		Write $n_i := |V(G_i)|$ and $k_i := |\varphi_i(G_i)|$. Note that $n_1 + n_2 = n + 3$ and $k_1 + k_2 \ge k + 3$. 
		By the induction hypothesis and Lemma~\ref{lem:A}, for each $i \in [2]$, there exists $E_i' \subseteq E(G_i) \setminus E(T)$ such that $|E_i'| \le n_i - k_i$ and $G_i - E_i'$ is acyclically colored by $\varphi_i$. Let $E' := E_1' \cup E_2'$. It is easy to prove that $G - E'$ is acyclically colored by $\varphi$ and
		$|E'| = |E_1'| + |E_2'| \le (n_1 - k_1) + (n_2 - k_2) \le n - k$. 
		
		Henceforth, we assume that $G$ has no separating triangle and thus $\delta(G) = 4, 5$. Fix $v \in V(G)$ such that $d_G(v) = \delta(G)$.
		Depending on the value of $\delta(G)$, we consider two cases.
		
		\smallskip

		{\bf Case 1:}  $d_G(v) = \delta(G) = 4$. 
		
		Let $v_1 v_2 v_3 v_4 v_1$ be the cycle induced by $N_G(v)$. Since $n > 6$ and $G$ has no separating triangle, we can assume that $v_1, v_3$ have no common neighbor other than $v, v_2, v_4$. 
		
		If $\varphi(v_1) \neq \varphi(v_3)$, we obtain $G'$ from $G$ by deleting $v$ and adding the edge $v_1 v_3$. Let $\varphi'$ be the restriction of $\varphi$ on $V(G')$. Denote $n' := |V(G')|$ and $k' := |\varphi'(V(G'))|$. Note that $G'$ is 4-connected, $n' = n - 1 \ge 6$ and $k' = k$ or $k - 1$. Moreover, if $k' = k - 1$, then $v$ is the only vertex that is colored by $\varphi(v)$ and hence no 2-colored cycle in $G$ contains $v$. By the induction hypothesis, there exists $E'' \subseteq E(G')$ such that $G' - E''$ is acyclically colored by $\varphi'$ and $|E''| = m(G', \varphi') \le n' - k'$. Define $S := \{v v_2\}$ if $k' = k$, and $S := \emptyset$ if $k' = k - 1$. Set $E' := (E'' \setminus \{v_1 v_3\}) \cup S$. One can readily show that $G - E'$ is acyclically colored by $\varphi$ and $|E'| \le n - k$. If $k = k' = 4$, we additionally require from the induction hypothesis that $|E''| \le n' - 5$, which yields in this case that $|E'| \le n - 5$. If $k = 4$ and $k' = k - 1$, then, suppose $\varphi(V(G)) = [4]$ and $\varphi(v) = 4$, one can deduce from Lemma~\ref{lem:n-3} that $G_{ij}$ are connected for all distinct $i, j \in [3]$ and hence prove in a similar way as in the proof of Theorem~\ref{thm:n-3} that $m(G, \varphi) = n - 5$.
		
		Assume $\varphi(v_1) = \varphi(v_3)$. First we prove that $m(G, \varphi) \le n- |\varphi(V(G))|$. Let $G'$ be from $G$ by contracting $v_1 v v_3$ to a new vertex $v'$ and denote the coloring induced from $\varphi$ by $\varphi'$ so that $\varphi(v') = \varphi(v_1)$. Denote $n' := |V(G')|$ and $k' := |\varphi'(V(G'))|$. We have $n' = n - 2 \ge 5$ and $k' = k$ or $k - 1$. By the induction hypothesis and Lemma~\ref{lem:A}, there exists $E'' \subseteq E(G') \setminus \{v'v_2, v'v_4\}$ such that $G' - E''$ is acyclically colored by $\varphi'$ and $|E''| = m(G', \varphi') \le n' - k'$. Note that any path joining $v_1, v_3$ in $G - \{v, v_2, v_4\}$ corresponds to a cycle containing $v'$ in $G'$ as $v_1, v_3$ have no common neighbor other than $v, v_2, v_4$. Define $S := \{v v_2\}$ if $k' = k$, and $S := \emptyset$ if $k' = k - 1$. Let $E' := E'' \cup \{v_1v_2\} \cup S$. It is clear that $|E'| \le n - k$ and $G - E'$ is acyclically colored by $\varphi$ as $v_1 v_2 v_3 v_4 v_1$ is the only cycle that is possibly 2-colored in $G - E'' - v$.
		
		It remains to show that if $k=4$, then $m(G, \varphi) \le n - 5$.
		If $\varphi(v_2) \neq \varphi(v_4)$, we take $E' := E''$ with $|E'| \le n' - 4 = n - 6$ and it is easy to show that $G - E'$ is acyclically colored by $\varphi$. So we assume that $\varphi(v_2) = \varphi(v_4)$. If $k' = 3$, then it follows from Theorem~\ref{thm:n-3} that $m(G, \varphi) \le n - 5$. So we assume $k' = 4$; in particular, $|E''| \le n' - 4$.
		If $|E''| = m(G', \varphi') \le n' - 5$, we take $E' := E'' \cup \{vv_2, v_1v_2\}$, so $|E'| = |E''| + 2 \le n - 5$ and $G - E'$ is acyclically colored by $\varphi$. This yields that $m(G, \varphi) \le |E'| \le n - 5$.
		
		Assume $m(G', \varphi') =|V(G')| - 4$.  As $|V(G')| > 4$, by the induction hypothesis, $G'$ is not 4-connected, and hence contains separating triangles. As $G$ is 4-connected, it follows that 
		each separating triangle of $G'$ contains $v'$ and separates $v_2$ and $v_4$; an example is given in Figure~\ref{fig:BC}.
		This implies that 
		$\mathcal{T}_{G'}$ is a path 
		$G_1' \dots G_t'$ ($t \ge 2$), with end-vertex $G'_1$ containing $v_2$, and the other end-vertex  $G'_t$ containing $v_4$.

		Denote by $\varphi'_i$ the restriction of $\varphi'$ on $V(G_i')$. By Lemma~\ref{lem:tritree} and Theorem~\ref{thm:n-3}, precisely one graph $G_i'$ from $\mathcal{V}_{G'}$ has $|\varphi'_i(V(G'_i))|=4$, $m(G'_i, \varphi_i)= |V(G'_i)| - 4$ and $|\varphi'_j(V(G_j))| = 3$ for all $j \in [t] \setminus \{i\}$.
		By the induction hypothesis, we know that $|V(G'_i)| \le 4$ and hence $G'_i$ is isomorphic to $K_4$.

		Note that $G_i'$ is not a leaf of $\mathcal{T}_{G'}$, for otherwise,  say $i = 1$, then $|\varphi'(V(G') \setminus \{v_2\}| = 3$. This implies that $\varphi(v_2) \neq \varphi(v_4)$, contradicting the above assumption.
		
		Thus $|\varphi'(V(G_j'))| = 3$ and $\{\varphi'(v'), \varphi'(v_2)\} \subset \varphi'(V(G_{j}'))$ for $j \in \{1,t\}$. As $G_i'$ is an internal vertex of $\mathcal{T}_{G'}$, we have $\varphi'(V(G_1')) \neq \varphi'(V(G_t'))$. Without loss of generality, we may assume that $\varphi'(V(G_1')) = [4] \setminus \varphi(v)$ and $\varphi'(V(G_t')) = \{\varphi(v), \varphi'(v'), \varphi'(v_2)\}$. Let $T$ be the separating triangle of $G'$ that is contained in $G_1'$. Write $V(T) := \{v', u, w\}$ such that $\varphi'(u) = \varphi'(v_2)$. Note that $\varphi'(w) \neq \varphi(v)$. Let $C$ be the cycle induced by the neighbors of $u$ in $G_1'$ (see Figure~\ref{fig:BC}(b) for an example) and $e_C$ be an arbitrary edge of $C$. By Lemma~\ref{lem:A}, we may require $E'' \subseteq E(G') \setminus (\{v'v_2, v'v_4\} \cup (E(C) \setminus \{e_C\}))$ as $\varphi'(v_2) = \varphi'(v_4) = \varphi'(u) \notin \varphi'(V(C))$, and hence $e_C \in E''$. Let $E' := (E'' \setminus \{e_C\}) \cup \{vv_2, v_1v_2\}$. We have $|E'| = |E''| + 1 \le n - 5$. It remains to show that $G - E'$ is acyclically colored by $\varphi$. Again, it is easy to show that $G - E' - e_C$ is acyclically colored by $\varphi$. Hence, any cycle $K$ which uses only two colors in $G - E'$ contains $e_C$ and the two colors 
		used in $K$ are $\varphi'(w), \varphi'(v')$. So $K$ 
		does not contain $v, v_2, v_4$. If $\{v_1, v_3\} \subset V(K)$, then after contracting the path $v_1 v v_3$, $K$ becomes the union of two edge-disjoint cycles in $(G'_{\varphi'(v')\varphi'(w)} - E'') + e_C$ (as $v_1, v_3$ have no other common neighbors than $v, v_2, v_4$), a contradiction. If $|\{v_1, v_3\} \cap V(K)| \le 1$, then $K$ corresponds to $C$. Since $C$ is a cycle separating $v_2$ and $v_4$ in $G'$, $K$ is a cycle separating $v_2$ and $v_4$ in $G$, which is however impossible since $v_2 v v_4$ is a path in $G$ not intersecting $K$. 
		
		\smallskip
		
		{\bf Case 2:}  $d_G(v) = \delta(G) = 5$. 
		
		Let $v_1v_2v_3v_4v_5v_1$ be the induced cycle on $N_G(v)$. 
		
		If $|\varphi(N_G(v))| = 3$, we may assume that $\varphi(v_1) = \varphi(v_3)$ and $\varphi(v_2) = \varphi(v_4)$. As $G$ is 4-connected and $\delta(G) = 5$, we may assume that $v_1, v_3$ have no common neighbor other than $v, v_2$. Let $G'$ be obtained from $G$ by contracting $v_1 v v_3$ to a new vertex $v'$. We do not distinguish edges from $E(G') \setminus \{v' v_2\}$ from their corresponding edges in $G$. Set $\varphi'(v') := \varphi(v_1)$ and $\varphi'(u) := \varphi(u)$ for all $u \in V(G') \setminus \{v'\}$. Denote $n' := |V(G')|$ and $k' := |\varphi'(V(G'))|$. We have $n' = n - 2$ and $k' = k$ or $k - 1$. By Lemma~\ref{lem:A} and the induction hypothesis, there exists $E'' \subseteq E(G') \setminus \{v'v_2\}$ such that $\varphi'$ is an acyclic coloring of $G' - E''$ and $|E''| = m(G', \varphi') \le n' - k'$. Set $S := \{v v_2\}$ if $k' = k$ and $S := \emptyset$ if $k' = k - 1$. Define $E' := E'' \cup S$. It is easy to show that $|E'| \le n - k - 1$ and $\varphi$ is an acyclic coloring of $G - E'$.
		
		If $|\varphi(N_G(v))| \ge 4$, we may assume that $\varphi(v_i) = i$ for each $i \in [4]$. Obtain $G'$ from $G$ by deleting $v$ and adding edges $v_1v_3, v_1v_4$. Let $\varphi'$ be the restriction of $\varphi$ on $V(G) \setminus \{v\}$. Denote $n' := |V(G')|$ and $k' := |\varphi'(V(G'))|$. We have $n' = n - 1$ and $k' = k$ or $k - 1$. By Lemma~\ref{lem:A} and the induction hypothesis, there exists $E'' \subseteq E(G') \setminus \{v'v_3, v'v_4\}$ such that $\varphi'$ is an acyclic coloring of $G' - E''$ and $|E''| = m(G', \varphi') \le n' - k'$. Set $S := \{v v_5\}$ if $k' = k$ and $S := \emptyset$ if $k' = k - 1$. Define $E' := E'' \cup S$. It is easy to show that $|E'| \le n - k$ and $\varphi$ is an acyclic coloring of $G - E'$. We remark that in this case we have $k > 4$, thus we do not need to consider the second statement.
	\end{proof}

	\begin{figure} [!ht]
		\centering
		\subfigure[]{\includegraphics[scale = 1.5]{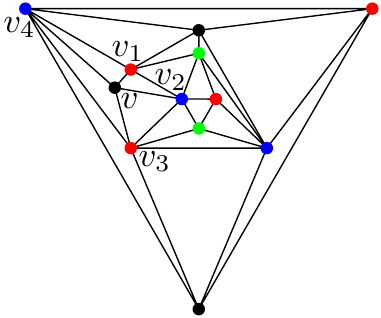}} \label{subfig:B}
		\hfil
		\subfigure[]{\includegraphics[scale = 1.5]{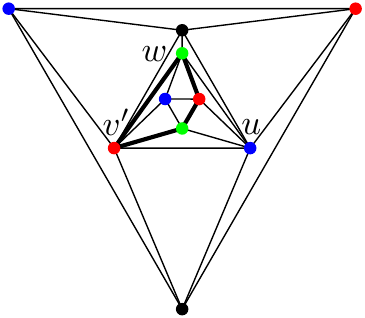}} \label{subfig:C}\\
		\caption{(a) A 4-colored plane triangulation $G$. (b) The plane triangulation $G'$ obtained from $G$ by contracting the path $v_1 v v_3$. The cycle $C$ consists of the thick edges.} 
		\label{fig:BC}
	\end{figure}
	
	The following corollary characterizes plane triangulations $G$ and colorings $\varphi$ that satisfy the eqaulities $m(G, \varphi) = n - 3$ and $m(G, \varphi) = n - 4$, respectively.
	
	\begin{corollary} \label{cor:n-3,n-4}
		Let $G$ be a plane triangulation on $n$ vertices and $\varphi$ be a coloring of $G$. Let $\mathcal{V}_{G} := \{G_1, \dots, G_t\}$ and $\varphi_i$ be the restriction of $\varphi$ on $V(G_i)$ for $i \in [t]$. We have that $m(G, \varphi) = n - 3$ if and only if $|\varphi(V(G))| = 3$; and $m(G, \varphi) = n - 4$ if and only if there exists $i \in [t]$ such that $G_i$ is isomorphic to $K_4$ and $|\varphi_j(V(G_j))| = 3$ for all $j \in [t] \setminus \{i\}$.
	\end{corollary}

	\section{Acyclic 2CC transerval and upper bounds for $m_k(G)$}
	\label{sec:subgraph}

	Let $G$ be a graph and $\varphi$ a coloring of $G$. We have shown upper bounds on $m(G, \varphi)$ when $G$ is a plane triangulation. In this section, we show that we can choose the 2CC  transversal $E'$ so that it induces a forest as well as extend the results to general planar graphs.

	\begin{definition}
		Let $G$ be a graph and $U \subseteq V(G)$. An edge set $E' \subseteq E(G)$ is \emph{$U$-acyclic} if the graph induced by $E'$ is a forest and contains no path joining two distinct vertices of $U$. With abuse of notation, we say an edge set is $H$-acyclic instead of $V(H)$-acyclic for any subgraph $H$ of $G$, and if $H$ is a graph induced by a single edge $e$, we write $e$-acyclic instead of $H$-acyclic.
	\end{definition}
	
	\begin{proposition} \label{pro:forest}
		Let $G$ be a plane triangulation and $\varphi$ be a proper coloring of $G$. For any facial cycle $F$ of $G$, there exists an $F$-acyclic $2$CC transversal $E_F$ with respect to $\varphi$. 
	\end{proposition}
	\begin{proof}
We prove by induction on $|V(G)|$. We shall assume $|V(G)| > \max\{6, |\varphi(V(G))|\}$ as the small cases can be readily verified.
		
		Suppose $G$ has some separating triangle $T$. Let $A_1$ and $A_2$ be the components of $G - T$, and for $i \in [2]$, $G_i$ be the subgraph of $G$ induced by $V(A_i) \cup V(T)$. Without loss of generality, assume that $F$ is a facial cycle of $G_1$. By the induction hypothesis, we have an $F$-acyclic 2CC transversal $E_F^1 \subseteq E(G_1)$ of $G_1$ and a $T$-acyclic 2CC transversal $E_T^2 \subseteq E(G_2)$ of $G_2$. It is easy to see that the edge set $E_F := E_F^1 \cup E_T^2$ is an $F$-acyclic 2CC transversal of $G$.
		
		Henceforth, we assume that $G$ has no separating triangle and thus $\delta(G) \ge 4$. Fix $v \in V(G) \setminus V(F)$ such that $d_G(v)=\delta(G) \le 5$. We consider two cases, depending on $d_G(v) = 4$ or $5$.

		\smallskip
		
		{\bf Case 1:}  $d_G(v) = 4$. 
		
		Let $v_1 v_2 v_3 v_4 v_1$ be the cycle induced by $N_G(v)$. Since $|V(G)| > 6$ and $G$ has no separating triangle, we can assume that $v_1, v_3$ have no common neighbor other than $v, v_2, v_4$. If $\varphi(v_1) \neq \varphi(v_3)$, we obtain $G'$ from $G$ by deleting $v$ and adding the edge $v_1 v_3$, and color it with the coloring $\varphi'$ induced from $\varphi$. Clearly, $F$ remains a facial cycle of $G'$. By the induction hypothesis, there exists an $F$-acyclic 2CC transversal $E_F' \subseteq E(G')$ of $G'$. Set $E_F := (E_F' \setminus \{v_1 v_3\}) \cup \{v v_2\}$. One can readily check that $E_F$ is an $F$-acyclic 2CC transversal of $G$.
		
		If $\varphi(v_1) = \varphi(v_3)$, obtain $G'$ from $G$ by contracting $v_1 v v_3$ to a new vertex $v'$ and denote the coloring induced from $\varphi$ by $\varphi'$ so that $\varphi(v') = \varphi(v_1)$. Let $E_F' \subseteq E(G')$ be an $F$-acyclic 2CC transversal of $G'$. Recall that $v_1, v_3$ have no common neighbor other than $v, v_2, v_4$, and hence any path joining $v_1, v_3$ in $G - \{v, v_2, v_4\}$ corresponds to a cycle containing $v'$ in $G'$. We construct $E_F$ as follows. \begin{itemize}
			\item If $E_F' \cap \{v'v_2, v'v_4\} = \emptyset$, then $v_1 v_2 v_3 v_4 v_1$ is the only cycle in $G - (E_F' \cup \{v v_2\})$ that possibly uses only two colors. We claim that there exists $j \in \{1, 3\}$ such that $E_F := E_F' \cup \{v v_2, v_j v_2\}$ induces a forest not connecting any distinct vertices from $V(F)$. Suppose it does not hold, then for each $j \in \{1, 3\}$, the graph induced by $E_F'$ in $G$ contains some path joining $v_j$ and $v_2$, or contains two disjoint paths each joining one vertex from $V(F)$ and one vertex of $v_j, v_2$. In any case, the graph induced by $E_F'$ in $G'$ contains some path joining two vertices from $V(F)$ or some cycle, a contradiction. As $G - E_F$ is acyclically colored by $\varphi$, $E_F$ is the desired edge set.
			\item If $E_F' \cap \{v'v_2, v'v_4\} = \{v'v_i\}$ for some $i \in \{2, 4\}$, set $E_F := (E_F' \setminus \{v'v_i\}) \cup \{vv_2, v_1v_i, v_3v_i\}$. Similarly to the previous case, it can be shown that $G - E_F$ is acyclically colored by $\varphi$ and the subgraph induced by $E_F$ has no cycle and no path joining distinct vertices from $V(F)$.
			\item If $\{v'v_2, v'v_4\} \subseteq E_F'$, then there is a unique path $P$ in $G'-E_F'$ joining $v'$ and $v_2$ using only colors $\varphi(v_1)$ and $\varphi(v_2)$. Therefore $P$ can be viewed as a path in $G - ((E_F' \setminus \{v' v_2, v' v_4\}) \cup E(v_1 v_2 v_3 v_4 v_1))$ connecting $v_2$ and $v_j$ for some $j \in \{1, 3\}$. Since $v_1, v_3$ have no common neighbor other than $v, v_2, v_4$ and the neighbor of $v'$ in $P$ is not $v_4$, the index $j$ is unique. Set $E_F := (E_F' \setminus \{v' v_2, v' v_4\}) \cup \{v v_2, v_j v_2, v_1 v_4, v_3 v_4\}$. Similarly to the previous cases, it is easy to show that $E_F$ is $F$-acyclic. It is left to show that $\varphi$ is an acyclic coloring of $G - E_F$. Suppose to the contrary that there is some 2-colored cycle $C$ in $G - E'$. It is not hard to see that $C$ contains $v_{4 - j}v_2$ but not $v_j$. Then $C - v_{4 - j} v_2$ is a path in $G' - E_F'$ connecting $v'$ and $v_2$ yet different from $P$, a contradiction. 
		\end{itemize}
		
		{\bf Case 2:}  $d_G(v) = 5$. 
		
		Let $v_1 v_2 v_3 v_4 v_5 v_1$ be the induced cycle on $N_G(v)$. If $|\varphi(N_G(v))| = 3$, we may assume that $\varphi(v_1) = \varphi(v_3)$ and $\varphi(v_2) = \varphi(v_4)$. Suppose $v_1, v_3$ have a common neighbor $u$ other than $v, v_2$ and $v_2, v_4$ have a common neighbor $u'$ other than $v, v_3$. Since $G$ has no separating triangle, $u = u'$ and $d_G(v_2) = d_G(v_3) = 4$. If $v_2$ or $v_3$ is not incident to $F$, we may revise our choice of $v$ so that $d_G(v) = 4$. Otherwise, $F$ is the cycle $u v_2 v_3 u$ and since $d_G(v) = 5$, there exists some vertex $w \in V(G) \setminus \{v, v_1, v_2, v_3, v_4, u\}$ such that $d_G(w) \le 5$; we may replace $v$ by $w$. Therefore, without loss of generality, we may assume that $v_1, v_3$ have no common neighbor other than $v, v_2$.
		
		Obtain $G'$ from $G$ by contracting $v_1 v v_3$ to a new vertex $v'$ and denote the coloring induced from $\varphi$ by $\varphi'$ so that $\varphi(v') = \varphi(v_1)$. It is clear that $F$ remains a facial cycle of $G'$. Let $E_F' \subseteq E(G')$ be an $F$-acyclic 2CC transversal of $G'$. We construct $E_F$ as follows. \begin{itemize}
			\item If $v' v_2 \in E_F'$, set $E_F := (E_F' \setminus \{v' v_2\}) \cup \{v v_2, v_1 v_2, v_2 v_3\}$. 
			\item If $v' v_2 \notin E_F'$, set $E_F := E_F' \cup \{v v_2\}$.
		\end{itemize} In both cases it is easy to show that $E_F$ is an $F$-acyclic 2CC transversal of $G$.
		
		If $|\varphi(N_G(v))| > 3$, we may assume that $\varphi(v_i) = i$ for each $i \in [4]$. Let $G'$ be the graph obtained from $G$ by deleting $v$ and adding edges $v_1v_3, v_1v_4$. Let $\varphi'$ be the restriction of $\varphi$ on $V(G) \setminus \{v\}$. Let $E_F'$ be an $F$-acyclic 2CC transversal of $G$. One can easily show that $E_F := (E_F' \setminus \{v_1v_3, v_1v_4\}) \cup \{vv_5\}$ is an $F$-acyclic 2CC transversal of $G$.
	\end{proof}
	
	We remark that the $F$-acyclic 2CC transversal $E_F$ found in Proposition~\ref{pro:forest} induces a forest of at least $|V(F)| = 3$ components and hence has size at most $|V(G)| - 3$. In fact, an $F$-acyclic 2CC transversal of the optimal size $m(G, \varphi)$ does exist due to the following observation.
	Note that for any edge set $E' \subseteq E(G)$, $G - E'$ is acyclically colored by a proper $k$-coloring $\varphi$ of $G$ if and only if $E(G) \setminus E'$ is an independent set of the direct sum of the graphic matroids of $G_{ij}$ ($i, j \in [k]$). This yields the following corollary.
	
	\begin{corollary} \label{cor:forest}
		Let $G$ be a plane triangluation, $\varphi$ be a proper coloring of $G$ and $F$ be a facial cycle of $G$. There exists an $F$-acyclic $2$CC transversal $E' \subseteq E(G)$ with $|E'| = m(G, \varphi)$.
	\end{corollary}
	
	Next, we generalize the results to planar graphs.

	\begin{theorem}
		Assume $G$ is a planar graph on $n$ vertices and $\varphi$ is a proper  coloring of $G$ with  $|\varphi(V(G))|=k$. Let $U \subseteq V(G)$  that induces a clique of size $|U| \le 3$. There exists a $U$-acyclic $2$CC transversal $E_U \subseteq E(G)$ with $|E_U| = m(G, \varphi) \le n - k$.
	\end{theorem}
	\begin{proof}
		We prove by induction on $n$. It clearly holds when $n \le k$. From now on we consider $n > k$.
		
		If $G$ has some separator $W \subset V(G)$ such that $|W| \le 3$ and $W$ induces a clique, let $A_1$ be a component of $G - W$ and $A_2$ the union of all other components. Denote by $G_i$ the subgraph of $G$ induced by $V(A_i) \cup W$ and by $\varphi_i$ the restriction of $\varphi$ on $V(G_i)$ ($i \in [2]$). Write $n_i := |V(G_i)|$ and $k_i := |\varphi_i(V(G_i)|$. We have $n_1 + n_2 = n - |W|$ and $k_1 + k_2 \ge k - |W|$. Without loss of generality, we require that $U \subseteq V(G_1)$. By the induction hypothesis, there exist a $U$-acyclic 2CC transversal $E_U'$ of $G_1$ with $|E_U'| \le n_1 - k_1$ and a $W$-acyclic 2CC transversal $E_W'$ of $G_2$ with $|E_W'| \le n_2 - k_2$. It is easy to show that $E_U := E_U' \cup E_W'$ is a $U$-acyclic 2CC transversal with $|E_U| \le n - k$.
		
		We assume that $G$ has no separator $W \subset V(G)$ such that $|W| \le 3$ and $W$ induces a clique. In particular, $G$ is 2-connected and every facial boundary of $G$ is a cycle. We add to $G$ as many edges as possible such that $\varphi$ remains as a proper coloring and $G$ remains as a plane graph. With abuse of notation, we call the new graph $G$. It suffices to prove the statement for the new graph $G$.
		
		If $G$ is a triangulation, we apply Theorem~\ref{thm:n-4} and Corollary~\ref{cor:forest} to conclude that $G$ has some $U$-acyclic 2CC transversal $E_U$ with $|E_U| = m(G, \varphi) \le n - k$.
		
		If any facial cycle of $G$ has a chord, then the end-vertices of the chord form a separator of $G$, contradicting our assumption.
		
		Assume $G$ is not a plane triangulation. As each facial cycle is an induced cycle, and any two non-adjacent vertices of a face are colored by the same color,  there exists a facial cycle $v_1v_2v_3v_4v_1$ in $G$ such that  $\varphi(v_1) = \varphi(v_3)$ and $\varphi(v_2) = \varphi(v_4)$. If $v_1, v_3$ have 3 common neighbors and $v_2, v_4$ have 3 common neighbors, then $G$ must be isomorphic to the plane graph obtained from the octahedron by deleting one vertex since we assume that $G$ has no separating triangle. One can easily verify that the statement holds for this graph. Thus, without loss of generality, we assume that $v_1, v_3$ have no common neighbor other than $v_2, v_4$. Let $G'$ be obtained from $G$ by identifying $v_1$ and $v_3$ as a new vertex $v'$ and $\varphi'$ be the coloring of $G'$ induced from $\varphi$. Denote $n' := |V(G')|$ and $k' := |\varphi'(V(G'))|$. We have $n' = n - 1$ and $k' = k$. Moreover, we can view $U$ as a vertex set of $G'$ since $U$ contains at most one of $v_1, v_3$. By the induction hypothesis, we have a $U$-acyclic 2CC transversal $E_U'$ of $G'$ with $|E_U'| = m(G', \varphi') \le n' - k'$. We construct $E_U$ as follows. Since the approach is similar to that in the proof of Proposition~\ref{pro:forest}, some details will be omitted.
		\begin{itemize}
			\item If $E_U' \cap \{v'v_2, v'v_4\} = \emptyset$, then there exists $j \in \{1, 3\}$ such that $E_U := E_U' \cup \{v_j v_2\}$ is $U$-acyclic.
			\item If $E_U' \cap \{v'v_2, v'v_4\} = \{v'v_i\}$ for some $i \in \{2, 4\}$, set $E_U := (E_U' \setminus \{v'v_i\}) \cup \{v_1v_i, v_3v_i\}$. 
			\item If $\{v'v_2, v'v_4\} \subseteq E_U'$, then there is a unique path $P$ in $G' - E_U'$ joining $v'$ and $v_2$ using only colors $\varphi(v_1)$ and $\varphi(v_2)$. We can view $P$ as a path in $G - ((E_F' \setminus \{v' v_2, v' v_4\}) \cup E(v_1 v_2 v_3 v_4 v_1))$ connecting $v_2$ and $v_j$ for some unique $j \in \{1, 3\}$. Set $E_U := (E_U' \setminus \{v' v_2, v' v_4\}) \cup \{v_j v_2, v_1 v_4, v_3 v_4\}$.
		\end{itemize} It is not hard to verify that the edge set $E_U$ constructed above is a $U$-acyclic 2CC transversal with $|E_U| \le n - k$. This completes the proof.
	\end{proof}

	\begin{corollary}
		Let $G$ be a planar graph on $n$ vertices. If $n \ge 5$, then $m_4(G) \le n-5$. If $G$ is $3$-colorable, then $m_3(G) \le n-3$. 
	\end{corollary}
	
	\begin{theorem}
		There are infinitely  many $4$-connected planar graphs $G$ with $m_4(G)=|V(G)|-5$, and infinitely many $3$-colorable planar graphs with $m_3(G)=|V(G)|-3$.
	\end{theorem}
	\begin{proof}
		It follows from Corollary \ref{cor:n-3,n-4} that for any 3-colorable plane triangulation $G$, $m_3(G)=|V(G)|-3$.
		
		Let $G$ be the 4-connected plane triangulation obtained by joining two independent vertices $u, v$ to every vertex of a cycle $C$ on $n - 2$ vertices with $n \ge 7$ odd. It is obvious that $G$ is not 3-colorable. Let $\varphi$ be any 4-coloring of $G$. Then, without loss of generality,  $\varphi(V(C)) = [3]$ and $\varphi(u) = \varphi(v) = 4$. For any $i \in [3]$, $G_{i4}$ is a connected plane graph with $|\varphi^{-1}(i)|$ faces, and for $i,j \in [3]$, $G_{ij}$ is acyclic. Therefore $m(G, \varphi) = \sum_{i \in [3]} (|\varphi^{-1}(i)| - 1) = n - 5$.
	\end{proof}

	\section{Upper bounds for $m'_k(G)$} \label{sec:m_k}
	
	In this section we study the problem of how many edges we need to remove from a planar graph in order to make it acyclic $k$-colorable for $k = 3, 4$.
	
	\begin{theorem}
		Let $G$ be a planar graph on $n$ vertices. We have $m_3(G) \le (13n - 42) / 10$ and $m_4(G) \le (3n - 12) / 5$.
	\end{theorem}
	\begin{proof}
		We first prove that $m_4(G) \le (3n - 12) / 5$. As every plane graph is a spanning subgraph of some plane triangulation, we may assume that $G$ is a plane triangulation on $n$ vertices. Let $\varphi: V(G) \rightarrow [5]$ be an acyclic 5-coloring of $G$. Without loss of generality, assume that \begin{align*}
			\sum_{v \in \varphi^{-1}(5)} (d_G(v) - 3) \le \frac15 \sum_{v \in V(G)} (d_G(v) - 3) = \frac{3n - 12}{5}.
		\end{align*} Let $v$ be any vertex in $\varphi^{-1}(5)$. Since the neighbors of $v$ span some cycle and $\varphi$ is acyclic, there exist $v_1, v_2, v_3 \in N_G(v)$ whose colors are pairwise distinct. Define $E_v$ to be the set of edges incident to $v$ other than $v v_1, v v_2$ and $v v_3$, and set $\varphi'(v)$ to be the color from $[4]$ other than $\varphi(v_1), \varphi(v_2), \varphi(v_3)$. To complete the construction, we set $E' := \bigcup_{v \in \varphi^{-1}(5)} E_v$ and set $\varphi'(u) := \varphi(u)$ for all $u \in \bigcup_{i \in [4]} \varphi^{-1}(i)$. It is readily to verify that $\varphi'$ is a proper 4-coloring of $G' := G - E'$ and $|E'| = \sum_{v \in \varphi^{-1}(5)} (d_G(v) - 3) \le \frac{3n - 12}{5}$. Suppose $\varphi'$ is not an acyclic coloring of $G'$, then there is a cycle $C$ contained in $\varphi'^{-1}(i) \cup \varphi'^{-1}(j)$ for some distinct $i, j \in [4]$. Note that $C$ cannot contain any $v \in \varphi^{-1}(5)$ since $v$ has precisely three neighbors of three different colors in $G'$. Therefore $C$ is contained in $G'[(\varphi'^{-1}(i) \cup \varphi'^{-1}(j)) \setminus \varphi^{-1}(5)] = G[\varphi^{-1}(i) \cup \varphi^{-1}(j)]$, a contradiction.
		
		This approach can be repeated to show that $m_3(G) \le (13n - 42) / 10$. More precisely, we may assume that \begin{align*}
			\sum_{v \in \varphi'^{-1}(4)} (d_{G'}(v) - 2) \le \frac14 \sum_{v \in V(G')} (d_{G'}(v) - 2) = \frac{4n - 12 - 2|E'|}{4}.
		\end{align*} It is not hard to see that for any $v \in V(G')$, $|\varphi'(N_{G'}(v))| \ge 2$. Let $v \in \varphi'^{-1}(4)$ and $v_1, v_2 \in N_{G'}(v)$ be of different colors. Define $E_v'$ to be the set of edges incident to $v$ other than $v v_1$ and $v v_2$, and set $\varphi''(v)$ to be the color from $[3]$ other than $\varphi(v_1), \varphi(v_2)$. Set $E'' := E' \cup \bigcup_{v \in \varphi'^{-1}(4)} E_v'$ and set $\varphi''(u) := \varphi'(u)$ for all $u \in \bigcup_{i \in [3]} \varphi'^{-1}(i)$. Again, it is readily to verify that $\varphi''$ is a proper 3-coloring of $G'' := G - E''$ and \begin{align*}
		|E''| = |E'| + \sum_{v \in \varphi'^{-1}(4)} (d_{G'}(v) - 2) \le \frac{13n - 42}{10}.
	\end{align*} Similarly as before, one can show that $\varphi''$ is an acyclic 3-coloring of $G''$ and hence the result follows.
\end{proof}

We remark that there exist infinitely many planar graphs $G$ on $n$ vertices so that $G - E'$ is not acyclically $4$-colorable for any $E' \subseteq E(G)$ with $|E'| < (n - 2) / 4$. Let $H$ be a 2-face-colorable triangulation and $\mathcal{T}$ be a family of $|E(H)| / 3$ edge-disjoint facial triangles of $H$. Let $G$ be obtained from $H$ by replacing each triangle from $\mathcal{T}$ by an octahedron. Therefore $E(G)$ is partitioned into $|E(H)| / 3$ octahedra, and $n = |V(H)| + |E(H)| = 4|V(H)| - 6$. As the octahedron is not acyclically 4-colorable, any $E' \subseteq E(G)$ satisfying that $G - E'$ is acyclically 4-colorable has size at least $|E(H)| / 3 = \frac{n - 2}{4}$.

\section*{Acknowledgments}
The research of On-Hei Solomon Lo was supported by a Postdoctoral Fellowship of Japan Society for the Promotion of Science and by Natural Sciences and Engineering
Research Council of Canada.
The research of Ben Seamone was supported by Natural Sciences and Engineering
Research Council of Canada.
The research of Xuding Zhu was supported by National Natural Science Foundation of China grant NSFC 11971438 and U20A2068.

\bibliographystyle{abbrv}
\bibliography{paper}

\end{document}